\documentclass[12pt]{amsart}
\usepackage{amsmath,amssymb,amsthm}
\usepackage{enumitem}
\theoremstyle{plain}%note that this is the default style (to learn)
\newtheorem{theorem}{Theorem}[section]

\newtheorem{lemma}[theorem]{Lemma}
\theoremstyle{definition}

\newtheorem{remark}[theorem]{Remark}

\newtheorem*{problem}{Inverse Problem}
\makeatletter
\@addtoreset{equation}{section}
\makeatother

\makeatletter
\newcommand{\Spvek}[2][r]{%
  \gdef\@VORNE{1}
  \left(\hskip-\arraycolsep%
    \begin{array}{#1}\vekSp@lten{#2}\end{array}%
  \hskip-\arraycolsep\right)}

\def\vekSp@lten#1{\xvekSp@lten#1;vekL@stLine;}
\def\vekL@stLine{vekL@stLine}
\def\xvekSp@lten#1;{\def\temp{#1}%
  \ifx\temp\vekL@stLine
  \else
    \ifnum\@VORNE=1\gdef\@VORNE{0}
    \else\@arraycr\fi%
    #1%
    \expandafter\xvekSp@lten
  \fi}
\makeatother

\begin{document}
\title[Inverse Problems for Jacobi Operators with mixed spectral data]
{Inverse Problems for Jacobi Operators with mixed spectral data\\}
\author{Burak Hat\.{i}no\u{g}lu}
\address{Department of Mathematics, Texas A{\&}M University, College Station,
	TX 77843, U.S.A.}
\email{burakhatinoglu@math.tamu.edu}

\subjclass[2010]{}

%\date{August 17, 2019}

\keywords{inverse spectral theory, semi-infinite Jacobi operators, Weyl $m$-function}

\begin{abstract}
We consider semi-infinite Jacobi matrices with discrete spectrum. We prove that the Jacobi operator can be uniquely 
recovered from one spectrum and subsets of another spectrum and norming constants corresponding to the first spectrum. 
We also solve this Borg-Marchenko-type problem under some conditions on two spectra, when missing part of the second spectrum and known 
norming constants have different index sets.
\end{abstract}
\maketitle

\section{\bf {Introduction}}

The Jacobi operator $J$ in the dense subset $c_{00}(\mathbb{N})$ of the Hilbert space $l^2(\mathbb{N})$ is the operator associated with the 
semi-infinite Jacobi matrix
\begin{equation*}
 \begin{pmatrix}
a_1 & b_1 & 0 & 0 \\
b_1 & a_2 & b_2 & \ddots \\
0  &  b_2 & a_3 & \ddots  \\
0 & \ddots & \ddots & \ddots
\end{pmatrix}
\end{equation*}
where $a_n \in \mathbb{R}$ and $b_n ~\textgreater~ 0$ for any $n \in \mathbb{N}$. The symmetric operator $J$ is closable and has deficiency indices 
(1,1) [limit point] or (0,0) [limit circle]. In the limit point case $\bar{J}$ is self-adjoint. However, in the limit circle case non self-adjoint 
operator $\bar{J}$ has a self-adjoint extension $J(g)$ uniquely determined by $g \in \mathbb{R} \cup \{\infty\}$ (see Section 2 and \cite{TES} Section 2.6). 
In both cases, a rank-one perturbation of a self-adjoint Jacobi operator can be seen as a change of the boundary condition at the origin for the 
corresponding Jacobi difference equation (see Section 2 and \cite{SW} Appendix).

Direct spectral problems aim to get spectral information from the sequences $\{a_n\}_{n \in \mathbb{N}}$ and $\{b_n\}_{n \in \mathbb{N}}$. In inverse 
spectral problems, the goal is to recover these sequences from spectral information, like the spectrum, the norming constants, the spectral measure 
or Weyl $m$-function.

The study of inverse problems for Jacobi operators is motivated both by pure mathematics, e.g. moment problems \cite{SIM} and physical applications, 
such as vibrating systems \cite{GLA,RAM}. 

Early inverse spectral problems for finite Jacobi matrices appear as discrete analogs of inverse spectral problems for Schr\"{o}dinger (Sturm-Liouville) 
equations
\begin{equation*}
 Lu = -u''+qu = zu,
\end{equation*}
where the potential $q\in L^1(0,\pi)$ is real-valued. Borg \cite{BOR} proved that an $L^1$-potential is uniquely recovered from two spectra, 
corresponding to various pairs of boundary conditions. Marchenko \cite{MAR} observed that the spectral measure (or Weyl-Titchmarsh $m$-function) 
uniquely recovers an $L^1$-potential. Another classical result is due to Hochstadt and Liebermann \cite{HL}, which says that if half of an 
$L^1$-potential is known, one spectrum recovers the whole. One can find the statements of these classical theorems and some other results from the 
inverse spectral theory of Schr\"{o}dinger operators e.g. in \cite{HAT} and references therein. 

Finite Jacobi matrix analogs of Borg's and Hochstadt and Lieberman's theorems were considered by Hochstadt \cite{HOC,HOC2,HOC3}, where the potential 
$q$ is replaced by the sequences $\{a_n\}_{n \in \mathbb{N}}$ and $\{b_n\}_{n \in \mathbb{N}}$. These classical theorems led to various other 
inverse spectral results on finite Jacobi matrices (see \cite{BD,DK,GES,GS,S,WW} and references therein), 
semi-infinite or infinite Jacobi matrices (see \cite{DKS,DKS2,DKS3,DS,ET,GES,GS,SW,SW2,TES2} and references therein), generalized Jacobi matrices 
(see \cite{D,D2} and references therein) and matrix-valued Jacobi operators (see \cite{CGR,GKM} and references therein). 
These problems can be divided into two groups. In Borg-Marchenko-type spectral problems, 
one tries to recover the sequences $\{a_n\}_{n \in \mathbb{N}}$ and $\{b_n\}_{n \in \mathbb{N}}$ from the spectral data. However, 
Hochstadt-Lieberman-type (or mixed) spectral problems recover the sequences $\{a_n\}_{n \in \mathbb{N}}$ and $\{b_n\}_{n \in \mathbb{N}}$ using a 
mixture of partial information on these sequences and spectral data.

Silva and Weder (\cite{SW} Theorem 3.3) proved Borg's two-spectra theorem for semi-infinite Jacobi matrices with a discrete spectrum. Later on Eckhardt 
and Teschl (\cite{ET} Theorem 5.2) considered infinite Jacobi matrix analog of Marchenko's result with the same discreteness of the spectrum assumption. 
Note that discreteness of the spectrum is an extra assumption in the limit point case.

Jacobi versions of Borg's and Hochstadt and Lieberman's theorems suggest that one spectrum gives exactly one half of the full spectral information 
required to recover the sequences $\{a_n\}_{n \in \mathbb{N}}$ and $\{b_n\}_{n \in \mathbb{N}}$. Let us recall the fact that in the case of discrete 
spectrum, the spectral measure is a discrete measure supported on the spectrum with the point masses given by the corresponding norming constants
(see \cite{ET} page 10). As follows from Jacobi analogs of Marchenko's theorem, the set of point masses of the 
spectral measure (or the set of norming constants) gives exactly one half of the full spectral information required to recover the sequences 
$\{a_n\}_{n \in \mathbb{N}}$ and $\{b_n\}_{n \in \mathbb{N}}$.

These observations allow us to formulate the following question:
\begin{problem}
Do one spectrum and partial information on another spectrum and the set of norming constants corresponding to the first spectrum recover 
the operator?
\end{problem}
This Borg-Marchenko-type problem can be seen as a combination of Silva and Weder's and Eckhardt and Teschl's results for semi-infinite Jacobi matrices.

In the present paper, we answer this question positively. Theorem \ref{IP1} solves this inverse spectral problem when given part of the norming constants 
corresponding to the first discrete spectrum matches with the missing part of the second discrete spectrum, i.e. they share the same index sets. 
In Theorem \ref{IP2} and Theorem \ref{IP3} we show that information of one of the boundary conditions can be replaced by any unknown eigenvalue from 
the second spectrum or any unknown norming constant corresponding to the first spectrum. In Theorems \ref{nIP1}, \ref{nIP2} and \ref{nIP3} we consider 
the same problems in the non-matching index sets case with some restrictions on the two spectra.

The paper is organized as follows. In Section 2 we recall necessary definitions and results from the spectral theory of Jacobi operators, namely 
self-adjoint extensions and rank-one perturbations of semi-infinite Jacobi operators, Weyl $m$-functions and the norming constants. In Section 3 after 
recalling Silva and Weder's characterization of the two spectra result, Theorem \ref{C2S}, and the two-spectra theorem, Theorem \ref{2ST}, we 
prove the inverse spectral problem mentioned above, Theorem \ref{IP1} along with Theorems \ref{IP2} and \ref{IP3}. In Section 4 we consider the same problems 
in the non-matching index sets case.

 \section{\bf {Preliminaries}}
 In this section we closely follow \cite{TES}.
 
 We consider the difference expression $\tau:l(\mathbb{N}) \rightarrow l(\mathbb{N})$
 \begin{align}
  (\tau f)_n &:= b_{n-1}f_{n-1} + a_nf_n + b_nf_{n+1}, \quad \quad n \in \mathbb{N} \text{\textbackslash} \{1\}\\
  (\tau f)_1 &:= a_1f_1 + b_1f_2
 \end{align}
where $a_n \in \mathbb{R}$, $b_n ~\textgreater~ 0$ for all $n \in \mathbb{N}$ and $l(\mathbb{N})$ is the set of complex valued sequences indexed by 
natural numbers. The difference expression $\tau$ is represented as the tridiagonal matrix
\begin{equation}
 \begin{pmatrix}
a_1 & b_1 & 0 & 0 & 0 \\
b_1 & a_2 & b_2 & 0 & \ddots \\
0  &  b_2 & a_3 & b_3 & \ddots  \\
0  & 0 & b_3 & a_4 & \ddots  \\
0 & \ddots & \ddots & \ddots & \ddots
\end{pmatrix}
\end{equation}
with respect to the canonical basis of $l^2(\mathbb{N})$. 

Let $c_z$,$s_z$ $\in l(\mathbb{N})$ be two fundamental solutions of the Jacobi difference equation 
\begin{equation}\label{DifEq}
 \tau u = zu \quad \quad u \in l(\mathbb{N}), ~ z\in \mathbb{C}, 
\end{equation}
satisfying the initial conditions 
\begin{align*}
&s_{z}(1) = 0, \quad s_{z}(2) = 1,\\
&c_{z}(1) = 1, \quad c_{z}(2) = 0.
\end{align*}
Since $c$ and $s$ are linearly independent, we write any solution $u$ of (\ref{DifEq}) as a linear combination of these two solutions
\begin{equation}\label{repu}
 u_z(n) = \frac{W_n(u_z,s_z)}{W_n(c_z,s_z)}c(n) - \frac{W_n(u_z,c_z)}{W_n(c_z,s_z)}s(n), 
\end{equation}
where $W$ is the Wronskian given by
\begin{equation*}
 W_n(f,g) = a(n)(f(n)g(n+1) - g(n)f(n+1)).
\end{equation*}
Note that the Wronskian of two solutions of (\ref{DifEq}) with the same $z$ is constant, so the coefficients of $c$ and $s$ in (\ref{repu}) are 
constant.

If $\{a_n\}$ and $\{b_n\}$ are bounded, then the Jacobi operator $J: l^2(\mathbb{N}) \rightarrow l^2(\mathbb{N})$ is defined as $Jf = \tau f$. 
However without the boundedness condition on $\{a_n\}$ or $\{b_n\}$, the operator $J$ is no longer defined on all of $l^2(\mathbb{N})$. Here one 
needs to introduce the minimal and maximal operators associated with $\tau$ as
\begin{align*}
 J_{min}:~ &\mathcal{D}(J_{min}) \rightarrow l^2(\mathbb{N}), \quad \quad J_{max}:~ \mathcal{D}(J_{max}) \rightarrow l^2(\mathbb{N})\\
 & f \mapsto \tau f \qquad \qquad \quad \quad \quad \qquad f \mapsto \tau f,
\end{align*}
where $\mathcal{D}(J_{min}) = c_{00}(\mathbb{N})$ and $\mathcal{D}(J_{max}) = \{ f \in l^2(\mathbb{N})~ |~ \tau f \in l^2(\mathbb{N})\}$. 
Green's formula implies that $J_{min}^* = J_{max}$ and 
\begin{align*}
 J_{max}^*=\overline{J_{min}} :~ &\mathcal{D}(J_{max}^*) \rightarrow l^2(\mathbb{N})\\
 & f \mapsto \tau f,
\end{align*}
where $\mathcal{D}(J_{max}^*) = \{ f \in J_{max} ~|~ \lim_{n \rightarrow \infty} W_n(\overline{f},g) = 0 ~,~ g \in J_{max}\}$ (\cite{TES}, Section 2.6).

In order to discuss self-adjoint extensions of the minimal operator we use limit point and limit circle classifications of $\tau$. The difference 
expression $\tau$ is called limit point ($l.p.$) if $s_{z_0} \notin l^2(\mathbb{N})$ for some $z_0 \in \mathbb{C}\text{\textbackslash}\mathbb{R}$ 
and limit circle ($l.c.$) otherwise.

The maximal operator $J_{max}$ is self-adjoint if and only if $\tau$ is $l.p.$ (\cite{TES}, Lemma 2.16). Therefore in the limit point case $J_{max}$ 
is a self adjoint extension of the minimal Jacobi operator $J_{min}$. 

If $\tau$ is limit circle, we define the set of boundary conditions at $\infty$ as
\begin{equation*}
  BC(\tau) = \{ v \in \mathcal{D}(J_{max})~|~ \lim_{n \rightarrow \infty} W_n(\bar{v} ,v) = 0,\lim_{n \rightarrow \infty} W_n(\bar{v},f) 
  \neq 0 \text{ for some } f \in \mathcal{D}(J_{max})\}.
\end{equation*}
Then for any $v \in BC(\tau)$, the operator
\begin{align*}
 J_v :~ &\mathcal{D}(v) \rightarrow l^2(\mathbb{N})\\
 & f \mapsto \tau f,
\end{align*}
is a self-adjoint extension of $J_{min}$, where $\displaystyle\mathcal{D}(v) = \{f \in \mathcal{D}(J_{max})~|~ 
\lim_{n \rightarrow \infty} W_n(v,f) = 0\}$ (\cite{TES}, Theorem 2.18). We parametrize self-adjoint extensions of $J_{min}$ in the limit circle case 
by defining
\begin{equation*}
 v_{\alpha}(n) = \cos(\alpha)c_0(n) + \sin(\alpha)s_0(n), \quad \quad \alpha \in [0,\pi)
\end{equation*}
and observing that different values of $\alpha$ give different extensions. Then all self-adjoint extensions of $J_{min}$ correspond to some $v_{\alpha}$ 
with unique $\alpha \in [0,\pi)$ (\cite{TES}, Lemma 2.20). Therefore in the limit circle case, following \cite{SW} we define $J(g) := J_v$ for 
$g \in \mathbb{R}\cup\{\infty\}$, where $g = \cot(\alpha)$ and $\alpha \in [0,\pi)$. In the limit point case, i.e. if $\overline{J_{min}}$ is 
self-adjoint, we let $J(g) := \overline{J_{min}}$ for all $g \in \mathbb{R}\cup\{\infty\}$. 

If $\tau$ is $l.c.$, i.e. $J_{min} \neq J_{min}^*$, then the spectrum of $J(g)$, denoted by $\sigma(J(g))$, is discrete (\cite{TES}, Lemma 2.19). 
Throughout the paper we assume $J(g)$ has a discrete spectrum, which is a restriction in the limit point case. Note that since the essential spectrum 
of a bounded Jacobi operator is always nonempty, discreteness of $\sigma(J(g))$ implies unboundedness of $J(g)$ (\cite{TES}, Section 3.2). 

We define the self-adjoint operator $J_h(g)$ by $J_h(g) := J(g) - h<\cdot,e_1>e_1$ for $h \in \mathbb{R}$, where $\{e_n\}_{n \in \mathbb{N}}$ is the 
canonical basis in $l^2(\mathbb{N})$. It is the rank-one perturbation of $J(g)$ by $h$. If we consider the operator $J(\beta,g)$ defined by the difference expression
\begin{equation*}
 (\tilde{\tau}f)_n := b_{n-1}f_{n-1} + a_nf_n + b_nf_{n+1}, \quad \quad n \in \mathbb{N}
\end{equation*}
with the boundary condition
\begin{equation*}
 f_1\cos\beta + f_0\sin\beta = 0, \quad \quad \beta \in (0,\pi),
\end{equation*}
then $J_h(g) = J(\beta,g)$ for $h = \cot\beta$. Hence $h$ can be seen as a boundary condition. Note that discreteness of $\sigma(J(g))$ implies 
discreteness of $\sigma(J_h(g))$ for any $h \in \mathbb{R}$. Moreover, $\sigma(J_{h_1}(g)) \cap \sigma(J_{h_2}(g)) = \emptyset$ if $h_1 \neq h_2$.

The Weyl $m$-function of $J_h(g)$, defined as 
\begin{equation*}
 m_h(z,g) := <e_1,(J_h(g) - z)^{-1}e_1>,
\end{equation*}
is a meromorphic Herglotz function, i.e. a meromorphic function with positive imaginary part on $\mathbb{C}_+$ satisfying 
$m(\overline{z}) = \overline{m(z)}$ (\cite{TES}, Section 2.1). By Neumann expansion for the resolvent
\begin{equation*}
 (J_h(g) - z)^{-1} = -\sum_{n=0}^{N-1}\frac{(J_h(g))^n}{z^{n+1}} + \frac{1}{z^N}(J_h(g))^N(J_h(g) - z)^{-1}, 
\end{equation*}
where $z \in \mathbb{C}\text{\textbackslash}\sigma(J_h(g))$ we get the following asymptotics of $m_h(z,g)$:
\begin{equation}\label{asym}
 m_h(z,g) = -\frac{1}{z} - \frac{a_1 - h}{z^2} - \frac{(a_1 - h)^2 + b_1^2}{z^3} + O(z^{-4}),
\end{equation}
as $z \rightarrow \infty$ for $\Im z \geq \epsilon$, $\epsilon~ \textgreater~ 0$ (\cite{TES}, Section 6.1).

Since the $m$-function $m_h(z,g)$ is Herglotz, if $\lambda$ is an isolated eigenvalue of $J_h(g)$, then $m_h(z,g)$ has a simple 
pole at $z=\lambda$ (\cite{TES}, Section 2.2). The norming constant corresponding to the eigenvalue $\lambda_k$ of $J_h(g)$ is  
\begin{equation*}
 \gamma_k(h) = \left(\sum_{n \in \mathbb{N}}|u_{\lambda_k}(n)|^2\right)^{-1},
\end{equation*}
where $u_{z} \in l^2(\mathbb{N})$ solves (\ref{DifEq}). The residue of $m_h(z,g)$ at the pole $\lambda_k$ is given by 
$-\gamma_k(h)$ (\cite{TES}, p.214).

One finds a detailed discussion of the spectral theory of Jacobi operators in \cite{TES}, which we have followed so far.

\section{\bf {Inverse spectral problems with mixed data}}
 
 We follow the enumeration introduced in \cite{SW} for enumerating the sequences of eigenvalues. Let $\{\lambda_n\}_{n}$ and 
 $\{\nu_n\}_{n}$ be a pair of discrete, interlacing, infinite real sequences and $M \subset \mathbb{Z}$. Then 
 $\lambda_n ~\textless~ \nu_n ~\textless~ \lambda_{n+1}$ for all $n \in M$, where 
 \begin{itemize}
  \item If $\inf_{n}\{\lambda_n\}_n = - \infty$ and $\sup_{n}\{\lambda_n\}_n = \infty$, then $M:=\mathbb{Z}$ and 
  $\nu_{-1} ~\textless~ 0 ~\textless~ \lambda_1$.\\
  \item If $0 ~\textless~ \sup_{n}\{\lambda_n\}_n ~\textless~ \infty$, then $M:=\{n\}_{n = -\infty}^{n_{max}}$, $n_{max} \geq 1$ and 
  $\nu_{-1} ~\textless~ 0 ~\textless ~\lambda_1$.\\
  \item If $\sup_{n}\{\lambda_n\}_n \leq 0$, then $M:=\{n\}_{n = -\infty}^{0}$.\\
  \item If $\inf_{n}\{\nu_n\}_n \geq 0$, then $M:=\{n\}_{n = 0}^{\infty}$.\\
  \item If $-\infty ~\textless~ \inf_{n}\{\nu_n\}_n ~\textless~ 0$, then $M:=\{n\}_{n = n_{min}}^{\infty}$, $n_{min} \leq -1$ and 
  $\nu_{-1} ~\textless~ 0 ~\textless~ \lambda_1$.\\
 \end{itemize}
 
Silva and Weder gave a characterization of the two spectra of $J(g)$ corresponding to different boundary conditions, if $J(g)$ has a discrete spectrum.
\begin{theorem}\label{C2S} \emph{(}\cite{SW} Theorem 3.4\emph{)} \emph{\textbf{(Characterization of two spectra)}} Given $h_1 \in \mathbb{R}$ and two infinite discrete sequence 
of real numbers $\{\lambda_n\}_{n\in M}$ and $\{\nu_n\}_{n\in M}$, there is a unique real number $h_2 ~\textgreater~ h_1$, a unique operator 
$J(g)$, and if $J_{min} \neq J_{min}^*$ also a unique $g \in \mathbb{R} \cup \{+\infty\}$, such that $\{\nu_n\}_{n\in M} = \sigma(J_{h_1}(g))$ and 
$\{\lambda_n\}_{n\in M} = \sigma(J_{h_2}(g))$ if and only if the following conditions are satisfied.
\begin{enumerate}
 \item $\{\lambda_n\}_{n\in M}$ and $\{\nu_n\}_{n\in M}$ interlace and, if $\{\lambda_n\}_{n\in M}$ is bounded from below,\\ 
 $\displaystyle\min_{n\in M}\{\nu_n\}_{n\in M} ~\textgreater~ \min_{n\in M}\{\lambda_n\}_{n\in M}$, while if $\{\lambda_n\}_{n\in M}$ is 
 bounded from above,\\
 $\displaystyle\max_{n\in M}\{\nu_n\}_{n\in M} ~\textgreater~ \max_{n\in M}\{\lambda_n\}_{n\in M}$.
 \item The following series converges 
 \begin{equation}\label{delta}
  \sum_{n \in M}(\nu_n - \lambda_n) = \Delta ~\textless~ \infty.
 \end{equation}
By condition (\ref{delta}) the product $\displaystyle \prod_{n \in M,n\neq k}\frac{\nu_n - \lambda_k}{\lambda_n - \lambda_k}$ is convergent, 
so define 
\begin{equation}
 \tau_k^{-1} := \frac{\nu_k - \lambda_k}{\Delta} \prod_{n \in M,n\neq k}\frac{\nu_n - \lambda_k}{\lambda_n - \lambda_k}, \quad \quad \quad 
 \forall k \in M.
\end{equation}
 \item The sequence $\displaystyle \{\tau_n\}_{n \in M}$ is such that, for $m=0,1,2,\dots$, the series 
 $\displaystyle \sum_{n \in M} \frac{\lambda_n^{2m}}{\tau_n}$ converges.
 \item If a sequence of complex numbers $\displaystyle \{\beta_n\}_{n \in M}$ is such that the series 
 $\displaystyle \sum_{n \in M} \frac{|\beta_n|^2}{\tau_n}$ converges and for $m=0,1,2,\dots$, 
 $\displaystyle \sum_{n \in M} \frac{\beta_n\lambda_n^m}{\tau_n} = 0$, then $\beta_n = 0$ for all $n \in M$.
\end{enumerate}
\end{theorem}

Silva and Weder also proved that the spectral data consisting of two discrete spectra and one of the boundary conditions uniquely determine 
the operator $J(g)$ and the other boundary condition.
\begin{theorem}\label{2ST} \emph{(}\cite{SW} Theorem 3.3\emph{)} \emph{\textbf{(Two-spectra theorem)}} Let $J(g)$ be the Jacobi operator with discrete spectrum, $h_1,h_2 \in \mathbb{R}$, 
$h_1 \neq h_2$, $\sigma(J_{h_1}(g)) = \{\lambda_n\}_{n \in M}$ and $\sigma(J_{h_2}(g)) = \{\nu_n\}_{n \in M}$. Then 
$\{\lambda_n\}_{n \in M}$, $\{\nu_n\}_{n \in M}$ and $h_1$(respectively $h_2$) 
uniquely determine the operator $J(g)$, $h_2$(respectively $h_1$) and if $J_{min} \neq J_{min}^*$, the boundary condition $g$ at infinity. 
\end{theorem}

Using Theorem \ref{C2S} and Theorem \ref{2ST} we prove our main inverse spectral result. The spectral data consists of one spectrum, 
a subset of another spectrum, the norming constants of the first spectrum for the missing part of the second spectrum and the two boundary conditions.

\begin{theorem}\label{IP1} \emph{\textbf{(Inverse problem-I)}} Let $J(g)$ be the Jacobi operator with discrete spectrum, 
$\sigma(J_{h_1}(g)) = \{\lambda_n\}_{n \in M}$, $\sigma(J_{h_2}(g)) = \{\nu_n\}_{n \in M}$ and $A$ be a subset of $M$. Then 
$\{\lambda_n\}_{n \in M}$, $\{\nu_n\}_{n \in M\text{\textbackslash} A}$, $\{\gamma_n(h_1)\}_{n \in A}$, $h_1$ and $h_2$ 
uniquely determine the operator $J(g)$, and if $J_{min} \neq J_{min}^*$, the boundary condition $g$ at infinity, where 
$\{\gamma_n(h_1)\}_{n \in M}$ are norming constants corresponding to $J_{h_1}(g)$. 
\end{theorem}

\begin{proof}
 The Weyl $m$-function $m_{h_1}$ can be represented in terms of $m_{h_2}$. Indeed, by the second resolvent identity and the definition 
 of the Weyl $m$-function
 \begin{align*}
  m_{h_1}(z,g) - m_{h_2}(z,g) &= < (T_{h_1}-T_{h_2})e_1,e_1 >\\
  &= < (T_{h_2})((h_1-h_2)<\cdot,e_1>)(T_{h_1})e_1,e_1>\\
  &= <(h_1-h_2)<T_{h_1}e_1,e_1>T_{h_2}e_1,e_1>\\
  &= (h_1 - h_2)m_{h_1}(z,g)m_{h_2}(z,g), 
 \end{align*}
where $T_{h} = (J_h(g) - zI)^{-1}$. Therefore 
\begin{equation}
 m_{h_2}(z,g) = \frac{m_{h_1}(z,g)}{1 - (h_2-h_1)m_{h_1}(z,g)}.
\end{equation}
Since $J(g)$ has discrete spectrum and $\displaystyle m_{h}(z,g) = \frac{m_{0}(z,g)}{1 - hm_{0}(z,g)}$, the poles of $m_h(z,g)$ are the eigenvalues 
of $J_h(g)$, given by the zeros of $1-hm_0(z,g)$ for any $h \in \mathbb{R}$. Hence 
\begin{equation}
 F(z,g):=\frac{m_{h_1}}{m_{h_2}} = \frac{1-h_2m_0}{1-h_1m_0}
\end{equation}
is a meromorphic function such that the zeros of $F$ are the eigenvalues of $J_{h_2}(g)$ and the poles of $F$ are the eigenvalues 
of $J_{h_1}(g)$. Moreover if $h_1-h_2 ~\textgreater~ 0$, then $F$ is a Herglotz function, since $m_0$ is a Herglotz function and 
\begin{equation*}
 F(z,g) = 1 + \frac{-1}{\frac{h_1}{h_1-h_2}+\frac{-1}{(h_1-h_2)m_0}}.
\end{equation*}
Let us assume $h_1 ~\textgreater~ h_2$. We consider the case $h_1 ~\textless~ h_2$ at the end of the proof, which will require minor 
changes.
Since $F$ is a meromorphic Herglotz function, by the infinite product representation of 
meromorphic Herglotz functions (\cite{LEV}, Theorem VII.1.1) and using the enumeration introduced above, $F$ can be represented as
\begin{equation}
 F(z,g) = C \frac{z-\nu_0}{z-\lambda_0}\prod_{n\in M, n \neq 0}\left(1-\frac{z}{\nu_{n}}\right)\left(1-\frac{z}{\lambda_n}\right)^{-1}, 
 \quad \quad C ~\textgreater~ 0.
\end{equation}
Recalling (\ref{delta}) and interlacing property of the two spectra $\{\lambda_n\}_{n\in M}$ and $\{\nu_n\}_{n\in M}$, one gets 
\begin{equation*}
 \Delta = \sum_{n \in M}|\nu_n - \lambda_n| \textless \infty,
\end{equation*}
and hence 
\begin{equation*}
 0 \quad \textless \prod_{n \in M, n \neq 0} \frac{\nu_n}{\lambda_n}\quad \textless \quad \infty.
\end{equation*}
Therefore
\begin{align*}
 \lim_{z \rightarrow \infty, \Im z\geq \epsilon} \frac{F(z,g)}{C} &= 
 \lim_{z \rightarrow \infty, \Im z\geq \epsilon} 
 \frac{z-\nu_0}{z-\lambda_0}\prod_{n\in M, n \neq 0}\left(1-\frac{z}{\nu_{n}}\right)\left(1-\frac{z}{\lambda_n}\right)^{-1} \\
 &= \lim_{z \rightarrow \infty, \Im z\geq \epsilon} 
 \frac{z-\nu_0}{z-\lambda_0}\prod_{n\in M, n \neq 0}\frac{\lambda_n}{\nu_n} 
 \prod_{n\in M, n \neq 0}\left(1+\frac{\nu_n - \lambda_n}{\lambda_n - z}\right)\\
 &= \prod_{n\in M, n \neq 0}\frac{\lambda_n}{\nu_n},
\end{align*}
for $\epsilon ~\textgreater~ 0$.
By (\ref{asym}) asymptotics of the $m$-function $m_0(z,g)$ implies $\displaystyle\lim_{z \rightarrow \infty, \Im z\geq \epsilon} m_0(z,g) = 0$ 
and by the definition of $F(z,g)$, we get $\displaystyle\lim_{z \rightarrow \infty, \Im z\geq \epsilon} F(z,g) = 1$. Therefore 
$\displaystyle C = \prod_{n\in M, n \neq 0}\frac{\nu_n}{\lambda_n}$ and 
\begin{equation}\label{repF}
 F(z,g) = \prod_{n\in M}\frac{z-\nu_n}{z-\lambda_n}.
\end{equation}
The residue of $F$ at $\lambda_k$ is given in terms of norming constant $\gamma_k(h_1)$. Indeed, 
\begin{equation*}
 \textrm{Res}(F,z=\lambda_k) = \textrm{Res}(\frac{m_{h_1}}{m_{h_2}},z=\lambda_k) = \textrm{Res}(1-(h_2-h_1)m_{h_1},z=\lambda_k) = \frac{-(h_1-h_2)}{\gamma_k(h_1)},
\end{equation*}
since $\gamma_k^{-1}(h_1) = -\textrm{Res}(m_{h_1},z=\lambda_k)$ for any $k \in M$. Recall that $\Delta=h_1-h_2$. Therefore 
\begin{equation}
 \frac{-1}{\gamma_n(h_1)} = \textrm{Res}\left(\frac{F}{\Delta},z=\lambda_n\right),
\end{equation}
i.e. the residues of $\frac{F(z,g)}{\Delta}$ are known at $\lambda_n$ for each $n \in A$.\\

At this step we can restate our claim in terms of $F$ as 
the set of poles, $\{\lambda_n\}_{n \in M}$, the set of zeros except the index set $A$, 
$\{\nu_n\}_{n \in M\text{\textbackslash} A}$, and the residues with the same index set $A$, 
$\{\textrm{Res}(\frac{F(z,g)}{\Delta},z=\lambda_n)\}_{n \in A}$ determine $F(z,g)$ uniquely. \\

Since $\{\nu_n-\lambda_n\}_{n\in M} \in l^1$, $F(z,g)$ has the representation $F = GH$, where 
$G(z,g) := \displaystyle\prod_{n \in A}\frac{z-\nu_n}{z-\lambda_n}$ and 
$H(z,g) := \displaystyle\prod_{n \in M\text{\textbackslash} A}\frac{z-\nu_n}{z-\lambda_n}$.\\

Note that for any $k \in A$, we know
\begin{equation}\label{Con1}
 Res\left(\frac{F}{\Delta},z=\lambda_k\right) = 
 \frac{\lambda_k-\nu_k}{\Delta}\prod_{n\in M,n\neq k}\frac{\lambda_k-\nu_n}{\lambda_k-\lambda_n}.
\end{equation}

In addition, for any $k \in A$, we also know 
\begin{equation}\label{Con2}
 H(\lambda_k) = \prod_{n \in M\text{\textbackslash}A}\frac{\lambda_k - \nu_n}{\lambda_k - \lambda_n}.
\end{equation}
Conditions (\ref{Con1}) and (\ref{Con2}) imply that for any $k\in A$, we know
\begin{equation*}
Res\left(\frac{G}{\Delta}, z=\lambda_k\right) = \frac{\textrm{Res}(\frac{F}{\Delta},z=\lambda_k)}{H(\lambda_k)}.
\end{equation*}
Note that zeros and poles of $G(z)$ are real and interlacing, and hence
\begin{equation*}
0 ~\textless~ arg(G(z)) = \sum_{n\in A}\left(arg(z-\nu_{n}) - arg(z-\lambda_{n})\right) ~\textless~ \pi
\end{equation*}
for any $z$ in the upper half plane, i.e. $G(z)$ is a 
meromorphic Herglotz function. Therefore by {\u C}ebotarev's theorem (\cite{LEV}, Theorem VII.1.2) $\frac{G(z)}{\Delta}$ has the following 
representation
\begin{equation}\label{CebTypeRep1}
 \frac{G(z)}{\Delta} = az + b + \sum_{n \in A}A_n\left(\frac{1}{\lambda_n-z}-\frac{1}{\lambda_n}\right),
\end{equation}
where $a \geq 0$ and $b \in \mathbb{R}$. Note that $A_k = -\textrm{Res}(\frac{G(z)}{\Delta}, z=a_k)$ for any $k \in A$, which means there are only two 
unknowns on the right hand side, namely constants $a$ and $b$.\\

On the upper half-plane $\frac{G(z)}{\Delta}$ converges to $\frac{1}{\Delta}$ as $z$ goes to infinity, since 
\begin{equation*}
\sum_{n \in A}|\nu_n - \lambda_n| \leq \sum_{n \in M}|\nu_n - \lambda_n| \textless \infty.
\end{equation*}
Let $t \in \mathbb{R}$. Then 
\begin{equation*}
 G(it) = \left[b + \sum_{n \in A}\left(\frac{\lambda_n A_n}{t^2+\lambda_n^2} - \frac{A_n}{\lambda_n}\right)\right] + 
 i\left[at + \sum_{n \in A}\frac{t A_n}{t^2+\lambda_n^2}\right],
\end{equation*}
and hence $a=0$ and $\displaystyle b = \frac{1}{\Delta} +\sum_{n \in A}\frac{A_n}{\lambda_n}$, since 
$\displaystyle \lim_{t \rightarrow +\infty} G(it) = \frac{1}{\Delta}$. Therefore
\begin{equation}\label{repG}
 G(z) = \frac{1}{\Delta} + \sum_{n \in A} \frac{A_n}{\lambda_n - z} = \frac{1}{h_1 - h_2} + \sum_{n \in A} \frac{A_n}{\lambda_n - z},
\end{equation}
so the right hand side of (\ref{repG}) is known.
This implies uniqueness of $G(z,g)$ and hence uniqueness of $\{\nu_n\}_{n\in A}$. 
After unique recovery of the two spectra $\sigma(J_{h_1}(g))$ and $\sigma(J_{h_2}(g))$, the operator $J$ is uniquely determined by 
Theorem \ref{2ST}.\\

If $h_2~\textgreater~ h_1$, then $\frac{1}{F(z,g)}$ is Herglotz instead of $F(z,g)$, so we get the infinite product representation 
\begin{equation}
 \frac{1}{F(z,g)} = \prod_{n \in M}\frac{z-\lambda_n}{z-\nu_n}.
\end{equation}
Note that $-F(z,g)$ is also a meromorphic Herglotz function. Therefore using similar arguments as $h_1~\textgreater~ h_2$ case, the function 
$G(z,g)$ definded as
\begin{equation*}
 G(z,g) := \frac{1}{h_1 - h_2} \prod_{n \in A}\frac{z - \nu_n}{z - \lambda_n}
\end{equation*}
is represented as
\begin{equation*}
 G(z,g) = \frac{1}{h_1-h_2} + \sum_{n \in A}\frac{A_n}{\lambda_n - z},
\end{equation*}
where $A_k = -\textrm{Res}(G,z=\lambda_k)$ for any $k \in A$. This implies uniqueness of $G(z,g)$ and hence uniqueness of $\{\nu_n\}_{n\in A}$. 
After unique recovery of the two spectra $\sigma(J_{h_1}(g))$ and $\sigma(J_{h_2}(g))$, the operator $J$ is uniquely determined by 
Theorem \ref{2ST}.
\end{proof}

\begin{remark}
 If we let $A = M$ in Theorem \ref{IP1}, then the spectral data becomes the Weyl $m$-function $m_{h_1}$(or the spectral measure corresponding to $h_1$) 
 and the boundary conditions $h_1$ and $h_2$. By letting $A = \emptyset$, we get the statement of the two-spectra theorem, Theorem \ref{2ST}.
\end{remark}

In our spectral data we can replace $h_1$ or $h_2$ with any eigenvalue of $J_{h_2}(g)$ from the index set $A$.

\begin{theorem}\label{IP2} \emph{\textbf{(Inverse problem-II)}} Let $J(g)$ be the Jacobi operator with discrete spectrum, 
$\sigma(J_{h_1}(g)) = \{\lambda_n\}_{n \in M}$, $\sigma(J_{h_2}(g)) = \{\nu_n\}_{n \in M}$ and $A$ be a subset of $M$. Then 
$\{\lambda_n\}_{n \in M}$, $\{\nu_n\}_{n \in M\text{\textbackslash} A}$, $\{\gamma_n(h_1)\}_{n \in A}$, $h_1$ (respectively $h_2$) 
and $\nu_m$ for some $m \in A$ uniquely determine the operator $J(g)$, $h_2$ (respectively $h_1$) and if $J_{min} \neq J_{min}^*$, the boundary 
condition $g$ at infinity, where $\{\gamma_n(h_1)\}_{n \in M}$ are norming constants corresponding to $J_{h_1}(g)$. 
\end{theorem}

\begin{proof}
 Following the proof of Theorem \ref{IP1} we get the infinite sum representation
 \begin{equation}\label{RepG2}
 \frac{G(z,g)}{\Delta} = \frac{1}{\Delta} + \sum_{n \in A} \frac{A_n}{\lambda_n - z} = 
 \frac{1}{h_1 - h_2} + \sum_{n \in A} \frac{A_n}{\lambda_n - z},
\end{equation}
 for the infinite product
\begin{equation*}
 G(z,g) := \displaystyle\prod_{n \in A}\frac{z-\nu_n}{z-\lambda_n}.
\end{equation*}
Now let us prove uniqueness of $G(z,g)$. Note that we know $\{\lambda_n\}_{n \in A}$, $\{-A_n\}_{n \in A}$ and $\nu_m$. 
Let the infinite product
\begin{equation*}
 \widetilde{G}(z,g) := \displaystyle\prod_{n \in A}\frac{z-\widetilde{\nu}_n}{z-\lambda_n}
\end{equation*}
share the same set of poles $\{\lambda_n\}_{n \in A}$ and the same residues $\{-A_n\}_{n \in A}$ at the corresponding poles with $G(z,g)$. 
In addition assume $G(z,g)$ and $\widetilde{G}(z,g)$ have the common zero $\nu_m$, i.e. $\widetilde{\nu}_m = \nu_m$. Let us also assume zeros and poles 
of $\widetilde{G}(z,g)$ satisfy asymptotic properties of Theorem \ref{C2S}. Then we know that $\widetilde{G}(z,g)$ has the infinite sum representation
\begin{equation}\label{RepG_2}
 \widetilde{G}(z,g) = \frac{1}{\widetilde{\Delta}} + \sum_{n \in A} \frac{A_n}{\lambda_n - z}
\end{equation}
Using representations (\ref{RepG2}) and (\ref{RepG_2}), the difference of $G(z,g)$ and $\widetilde{G}(z,g)$ is a real constant, which is zero since $G(\nu_m,g)=\widetilde{G}(\nu_m,g)$. This implies uniqueness 
of $G(z,g)$ and hence uniqueness of $\{\nu_n\}_{n\in A}$. After unique recovery of the two spectra $\sigma(J_{h_1}(g))$ and $\sigma(J_{h_2}(g))$, 
the operator $J$ is uniquely determined by Theorem \ref{2ST}.
\end{proof}

In the spectral data of Theorem \ref{IP1} we can also replace $h_1$ or $h_2$ with any norming constant of $J_{h_1}(g)$ 
outside the index set $A$.

\begin{theorem}\label{IP3} \emph{\textbf{(Inverse problem-III)}} Let $J(g)$ be the Jacobi operator with discrete spectrum, 
$\sigma(J_{h_1}(g)) = \{\lambda_n\}_{n \in M}$, $\sigma(J_{h_2}(g)) = \{\nu_n\}_{n \in M}$ and $A$ be a subset of $M$. Then 
$\{\lambda_n\}_{n \in M}$, $\{\nu_n\}_{n \in M\text{\textbackslash} A}$, $\{\gamma_n(h_1)\}_{n \in A}$, $h_1$ (respectively $h_2$) 
and $\gamma_m(h_1)$ for some $m \in M\text{\textbackslash} A$ uniquely determine the operator $J(g)$, $h_2$ (respectively $h_1$) and if 
$J_{min} \neq J_{min}^*$, the boundary condition $g$ at infinity, where $\{\gamma_n(h_1)\}_{n \in M}$ are norming constants corresponding 
to $J_{h_1}(g)$. 
\end{theorem}

\begin{proof}
 Let us define the index set $A':=A\cup\{m\}$. Then following the proof of Theorem \ref{IP1} and redefining $G$ and $H$ as 
\begin{equation*}
 G(z,g) := \displaystyle\prod_{n \in A'}\frac{z-\nu_n}{z-\lambda_n}
\end{equation*}
and
\begin{equation*}
 H(z,g) := \displaystyle\prod_{n \in M\text{\textbackslash} A'}\frac{z-\nu_n}{z-\lambda_n}
\end{equation*}
we get 
\begin{equation}\label{RepG3}
 G(z,g) = \frac{1}{\Delta} + \sum_{n \in A'} \frac{A_n}{\lambda_n - z} = 
 \frac{1}{h_1 - h_2} + \sum_{n \in A'} \frac{A_n}{\lambda_n - z}.
\end{equation}

Now let us prove uniqueness of $G(z,g)$. Note that we know $\{\lambda_n\}_{n \in A'}$, $\{-A_n\}_{n \in A'}$ and $\nu_m$. 
Let the infinite product
\begin{equation*}
 \widetilde{G}(z,g) := \displaystyle\prod_{n \in A'}\frac{z-\widetilde{\nu}_n}{z-\lambda_n}
\end{equation*}
share the same set of poles $\{\lambda_n\}_{n \in A'}$ and the same residues $\{-A_n\}_{n \in A'}$ at the corresponding poles with $G(z,g)$. 
In addition $G(z,g)$ and $\widetilde{G}(z,g)$ have the same zero $\nu_m$, i.e. $\widetilde{\nu}_m = \nu_m$. Let us also assume zeros and poles of 
$\widetilde{G}(z,g)$ satisfy asymptotic properties of Theorem \ref{C2S}. Then we know that $\widetilde{G}(z,g)$ has the infinite sum representation
\begin{equation}\label{RepG_3}
 \widetilde{G}(z,g) = \frac{1}{\widetilde{\Delta}} + \sum_{n \in A'} \frac{A_n}{\lambda_n - z}
\end{equation}
Using representations (\ref{RepG3}) and (\ref{RepG_3}), the difference of $G(z,g)$ and 
$\widetilde{G}(z,g)$ is a real constant, which is zero since $G(\nu_m,g)=\widetilde{G}(\nu_m,g)$. This implies uniqueness of $G(z,g)$ 
and hence uniqueness of $\{\nu_n\}_{n\in A}$. After unique recovery of the two spectra $\sigma(J_{h_1}(g))$ and $\sigma(J_{h_2}(g))$, 
the operator $J$ is uniquely determined by Theorem \ref{2ST}.
\end{proof}

\section{\bf {Non-matching index sets}}

If the known norming constants of $J_{h_1}(g)$ and unknown eigenvalues of $J_{h_2}(g)$ have different index sets, one needs 
some control over eigenvalues of $J_{h_1}(g)$ corresponding to the known norming constants and unknown part of the spectrum 
$\sigma(J_{h_2}(g))$. In this case we get a {\u C}ebotarev type representation result. Before the statement, let us clarify the notation 
we use. For any subsequence $\displaystyle\{\lambda_{k_n}\}_{n\in \mathbb{N}} \subset \sigma(J_{h_1}(g))$ and 
$\{\nu_{l_n}\}_{n\in \mathbb{N}} \subset \sigma(J_{h_2}(g))$, by $A_{k_n,m}$ and $A_{k_n}$ we denote the residues at 
$\lambda_{k_n}$ of partial and infinite products, respectively, consisting of these subsequences: 
\begin{align*}
 A_{k_n,m} &:= \frac{\lambda_{k_n}}{\nu_{l_n}}(\lambda_{k_n}-\nu_{l_n})\prod_{1\leq j\leq m, j\neq n}
 \frac{\lambda_{k_j}}{\nu_{l_j}}\frac{\lambda_{k_n}-\nu_{l_j}}{\lambda_{k_n}-\lambda_{k_j}},\\
 A_{k_n} &:= \frac{\lambda_{k_n}}{\nu_{l_n}}(\lambda_{k_n}-\nu_{l_n})\prod_{j\in\mathbb{N},j\neq n}
 \frac{\lambda_{k_j}}{\nu_{l_j}}\frac{\lambda_{k_n}-\nu_{l_j}}{\lambda_{k_n}-\lambda_{k_j}}.
\end{align*}
Note that these subsequences are ordered according to their indices, i.e. $\displaystyle a_{k_n} ~\textless ~a_{k_{n+1}}$ and 
$\displaystyle b_{l_n} ~\textless ~b_{l_{n+1}}$ for any $n \in \mathbb{N}$. This follows from the fact that the two spectra 
are both real and discrete. Also note that if the spectrum $\sigma(J_h(g))$ is unbounded from both sides, i.e. $\inf M = -\infty$ and 
$\sup M = \infty$ in the enumeration, then $\{k_n\}_n$ and $\{l_n\}_n$ should be indexed by $\mathbb{Z}$ instead of $\mathbb{N}$. However, 
wlog we index them by $\mathbb{N}$.

\begin{lemma}\label{lmm} Let $\displaystyle\{\lambda_{k_n}\}_{n\in \mathbb{N}} \subset \sigma(J_{h_1}(g))$, 
$\displaystyle \{\nu_{l_n}\}_{n\in \mathbb{N}} \subset \sigma(J_{h_2}(g))$ such that
\begin{itemize}
 \item $\displaystyle \inf_{n \in \mathbb{N}} |\nu_{l_n} - \lambda_{k_n}| ~\textgreater~ 0$ and
 \item $\displaystyle \frac{A_{k_n,m} - A_{k_n}}{\lambda_{k_n}} \in l^1$ for all $n \in \{1,\dots,m\}$, for all $m \in \mathbb{N}$.
\end{itemize}
Then the infinite product
\begin{equation*}
 \mathcal{G}(z) := \prod_{n\in \mathbb{N}} \left(\frac{z}{\nu_{l_n}}-1\right)\left(\frac{z}{\lambda_{k_n}}-1\right)^{-1}
\end{equation*}
is represented as
\begin{equation}\label{CebTypeRep2}
  \mathcal{G}(z) = az + b + \sum_{n\in \mathbb{N}} A_{k_n}\left(\frac{1}{z-\lambda_{k_n}}+\frac{1}{\lambda_{k_n}}\right),
\end{equation}
where $a, b$ are real numbers, $A_{k_n}$ is the residue of $\mathcal{G}(z)$ at the point $z=\lambda_{k_n}$ and the product converges 
normally on $\displaystyle \mathbb{C} \text{\textbackslash} \cup_{n \in \mathbb{N}} \lambda_{k_n}$.
\end{lemma}

\begin{proof}
 Let $p(z)$ be the difference of $\mathcal{G}(z)$ and the infinite sum in the right hand side of (\ref{CebTypeRep2}). Then, $p(z)$ is an entire function, 
 since the infinite product and the infinite sum share the same set of poles with equivalent degrees and residues. We represent partial products of 
 $\mathcal{G}(z)$ as partial sums:
 \begin{equation*}
  \prod_{n=1}^m \left(\frac{z}{\nu_{l_n}}-1\right)\left(\frac{z}{\lambda_{k_n}}-1\right)^{-1} = 
  \sum_{n=1}^m A_{k_n,m}\left(\frac{1}{z-\lambda_{k_n}}+\frac{1}{\lambda_{k_n}}\right) + 1,
 \end{equation*}
where $A_{k_n,m}$ is the residue of the partial product at $a_{k_n}$.\\

If $\sigma(J_{h}(g))$ is not bounded above, then let ${C_n}$ be the circle with radius $\nu_{l_n}$ centered at the origin for 
$\nu_{l_n} ~\textgreater~ 0$. If $\sigma(J_{h}(g))$ is bounded above, then let ${C_n}$ be the circle with radius $|\nu_{l_n}|$ centered 
at the origin for $\nu_{l_n} ~\textless~ 0$.
This sequence of circles satisfy following properties for sufficiently large $n$:
 
 \begin{itemize}
 \item $C_n$ omits all the poles $\lambda_{k_n}$.
 \item Each $C_n$ lies inside $C_{n+1}$.
 \item The radius of $C_n$, $|\nu_{l_n}|$ diverges to infinity as $n$ goes to infinity.
\end{itemize}
Then,
\begin{align*}
\max_{z\in C_t}\left|\frac{p(z)-1}{\nu_{l_t}}\right| &= \max_{z\in C_t}\left|\frac{\mathcal{G}(z)-1 - 
	\sum_{n\in \mathbb{N}} A_{k_n}\left(\frac{1}{z-\lambda_{k_n}}+\frac{1}{\lambda_{k_n}}\right)}{\nu_{l_t}}\right|\\
&= \frac{1}{|\nu_{l_t}|} \max_{z\in C_t}\lim_{m \to \infty} \left| \sum_{n=1}^m A_{k_n,m}\left(\frac{1}{z-\lambda_{k_n}}+
\frac{1}{\lambda_{k_n}}\right) - \sum_{n=1}^m A_{k_n}\left(\frac{1}{z-\lambda_{k_n}}+\frac{1}{\lambda_{k_n}}\right)\right|\\
&= \lim_{m \to \infty} \frac{1}{|\nu_{l_t}|}\max_{z\in C_t}\left| \sum_{n=1}^m (A_{k_n,m}-A_{k_n})\frac{z}{\lambda_{k_n}(z-\lambda_{k_n})}\right|\\
&\leq \lim_{m \to \infty} \frac{1}{|\nu_{l_t}|}\sum_{n=1}^m |A_{k_n,m}-A_{k_n}|\frac{\nu_{l_t}}{\lambda_{k_n}|\nu_{l_t}-\lambda_{k_n}|}
~\textless~ \infty.
\end{align*}
Finiteness of the last limit follows from the two assumptions on the sequences $\displaystyle\{\lambda_{k_n}\}_{n\in \mathbb{N}}$ and 
$\displaystyle \{\nu_{l_n}\}_{n\in \mathbb{N}}$. Therefore $|p(z)-1|\leq C|z|$ on the circle $C_t$ for 
any $t \in \mathbb{N}$, where $C$ is a positive real number. By the maximum modulus theorem and the entireness of $p(z)$, we conclude 
that $p(z)$ is a polynomial of at most first degree.
\end{proof}

Using the {\u C}ebotarev type representation (\ref{CebTypeRep2}) we prove our inverse spectral results in non-matching index sets case 
with some additional convergence criterion on the two spectra. Theorems \ref{nIP1}, \ref{nIP2} and \ref{nIP3} are non-matching index sets versions of 
Theorems \ref{IP1}, \ref{IP2} and \ref{IP3} respectively.

\begin{theorem}\label{nIP1} \emph{\textbf{(Inverse Problem IV)}} Let $J(g)$ be the Jacobi operator with discrete spectrum, 
$\sigma(J_{h_1}(g)) = \{\lambda_n\}_{n \in M}$, $\sigma(J_{h_2}(g)) = \{\nu_n\}_{n \in M}$ and 
$\displaystyle\{\lambda_{k_n}\}_{n\in \mathbb{N}} \subset \sigma(J_{h_1}(g))$, 
$\displaystyle \{\nu_{l_n}\}_{n\in \mathbb{N}} \subset \sigma(J_{h_2}(g))$ such that
\begin{itemize}
 \item $\displaystyle \inf_{n \in \mathbb{N}} |\nu_{l_n} - \lambda_{k_n}|~ \textgreater ~0$,\\
 \item $\displaystyle \frac{A_{k_n,m} - A_{k_n}}{\lambda_{k_n}} \in l^1$ for all $n \in \{1,\dots,m\}$, for all $m \in \mathbb{N}$,\\
 \item $\{\lambda^{-1}_n\}_{n \in M} \in l^1$,\\
 \item $\{\nu^{-1}_n\}_{n \in M\text{\textbackslash}\{l_n\}_{n \in \mathbb{N}}} \in l^1$ and\\
 \item $\displaystyle 0~ \textless ~\prod_{n \in \mathbb{N}} \frac{\nu_{l_n}}{\lambda_{k_n}}~ \textless ~\infty$.
\end{itemize} 
Then $\{\lambda_n\}_{n \in M}$, $\{\nu_n\}_{n \in M}\text{\textbackslash}\{\nu_{l_n}\}_{n\in \mathbb{N}}$, 
$\{\gamma_{k_n}(h_1)\}_{n \in \mathbb{N}}$, $h_1$ and $h_2$ uniquely determine the operator $J(g)$, and if $J_{min} \neq J_{min}^*$, 
the boundary condition $g$ at infinity, where $\{\gamma_n(h_1)\}_{n \in M}$ are norming constants corresponding to $J_{h_1}(g)$. 
\end{theorem}

 \begin{proof}
  As we discussed in the proof of Theorem \ref{IP1} wlog we assume $h_1 ~\textgreater~ h_2$. Recall that in this case 
  \begin{equation*}
   \Delta := h_1 - h_2 = \sum_{n \in M} \nu_n - \lambda_n ~\textless~ \infty. 
  \end{equation*}
  Let us define 
  \begin{equation}
   \mathcal{F}(z,g) := \frac{1}{\Delta}\prod_{n \in M}\frac{\lambda_n}{\nu_n}\prod_{n \in M}\frac{z-\nu_n}{z-\lambda_n}.
  \end{equation}
 Note that we assume $\{\lambda^{-1}_n\}_{n \in M} \in l^1$, $\{\nu^{-1}_n\}_{n \in M\text{\textbackslash}\{l_n\}_{n \in \mathbb{N}}} \in l^1$ and 
 $\displaystyle 0~ \textless ~\prod_{n \in \mathbb{N}} \frac{\nu_{l_n}}{\lambda_{k_n}}~ \textless ~\infty$. Therefore $\mathcal{F}(z,g)$ has the 
 representation $\mathcal{F} = \mathcal{G}\mathcal{H}$, where 
 \begin{equation*}
  \mathcal{G}(z,g) := \frac{1}{\Delta}\prod_{n \in \mathbb{N}}\frac{\nu_{l_n}-z}{\nu_{l_n}}\frac{\lambda_{k_n}}{\lambda_{k_n}-z}
  = \frac{1}{\Delta}\prod_{n \in \mathbb{N}}\frac{\lambda_{k_n}}{\nu_{l_n}}\prod_{n \in \mathbb{N}}\frac{z-\nu_{l_n}}{z-\lambda_{k_n}}
 \end{equation*}
 and
 \begin{equation*}
  \mathcal{H}(z,g) := \prod_{n \in M\text{\textbackslash}\{l_n\}}\frac{\nu_n-z}{\nu_n}
  \prod_{n \in M\text{\textbackslash}\{k_n\}}\frac{\lambda_n}{\lambda_n-z}.
 \end{equation*}
By (\ref{repF})
\begin{equation*}
 Res\left(\left[\prod_{n \in M}\frac{\nu_n}{\lambda_n}\right]\mathcal{G}(z,g), z = \lambda_k\right) = \frac{1}{\mathcal{H}(\lambda_k)\gamma_k(h_1)},
\end{equation*}
so we know the residues of the infinite product $\displaystyle\left(\prod_{n \in M}\frac{\nu_n}{\lambda_n}\right)\mathcal{G}(z,g)$ at $\lambda_k$ for 
any $k \in \{k_n\}_{n \in \mathbb{N}}$. This inifinite product has the representation
\begin{equation*}
 \left(\prod_{n \in M}\frac{\nu_n}{\lambda_n}\right)\mathcal{G}(z,g) = 
 \left(\frac{1}{\Delta}\prod_{n \in M}\frac{\nu_n}{\lambda_n}\prod_{n \in \mathbb{N}}\frac{\lambda_{k_n}}{\nu_{l_n}}\right)
 \prod_{n \in \mathbb{N}}\frac{z-\nu_{l_n}}{z-\lambda_{k_n}} = C\prod_{n \in \mathbb{N}}\frac{z-\nu_{l_n}}{z-\lambda_{k_n}}.
\end{equation*}
Let us observe that $C$ is a real constant depending only on $\{\lambda_n\}_{n \in M \text{\textbackslash}\{k_n\}}$, 
$\{\nu_n\}_{n \in M \text{\textbackslash}\{l_n\}}$, $h_1$ and $h_2$, so we also know $C$. From Lemma \ref{lmm} we get the {\u C}ebotarev type 
representation 
\begin{equation*}
 \left(\prod_{n \in M}\frac{\nu_n}{\lambda_n}\right)\mathcal{G}(z,g) = C\prod_{n \in \mathbb{N}}\frac{z-\nu_{l_n}}{z-\lambda_{k_n}} 
 = az + b + \sum_{n\in \mathbb{N}} A_{k_n}\left(\frac{1}{z-\lambda_{k_n}}+\frac{1}{\lambda_{k_n}}\right).
\end{equation*}
Using similar arguments as in the proof of Theorem \ref{IP1} one finds that $a = 0$ and 
$b = C - \displaystyle\sum_{n\in \mathbb{N}}\frac{A_{k_n}}{\lambda_{k_n}}$ and hence 
\begin{equation}\label{RepG4}
 \left(\prod_{n \in M}\frac{\nu_n}{\lambda_n}\right)\mathcal{G}(z,g) = C + \sum_{n\in \mathbb{N}}\frac{A_{k_n}}{z-\lambda_{k_n}}.
\end{equation}
The right hand side of (\ref{RepG4}) is known. This implies uniqueness of $\mathcal{G}(z,g)$ and hence uniqueness of 
$\{\nu_{l_n}\}_{n \in \mathbb{N}}$. After unique recovery of the two spectra $\sigma(J_{h_1}(g))$ and $\sigma(J_{h_2}(g))$, 
the operator $J$ is uniquely determined by Theorem \ref{2ST}.
\end{proof}

\begin{theorem}\label{nIP2} \emph{\textbf{(Inverse Problem V)}} Let $J(g)$ be the Jacobi operator with discrete spectrum, 
$\sigma(J_{h_1}(g)) = \{\lambda_n\}_{n \in M}$, $\sigma(J_{h_2}(g)) = \{\nu_n\}_{n \in M}$ and 
$\displaystyle\{\lambda_{k_n}\}_{n\in \mathbb{N}} \subset \sigma(J_{h_1}(g))$, 
$\displaystyle \{\nu_{l_n}\}_{n\in \mathbb{N}} \subset \sigma(J_{h_2}(g))$ such that
\begin{itemize}
 \item $\displaystyle \inf_{n \in \mathbb{N}} |\nu_{l_n} - \lambda_{k_n}| ~\textgreater~ 0$,\\
 \item $\displaystyle \frac{A_{k_n,m} - A_{k_n}}{\lambda_{k_n}} \in l^1$ for all $n \in \{1,\dots,m\}$, for all $m \in \mathbb{N}$,\\
 \item $\{\lambda^{-1}_n\}_{n \in M} \in l^1$,\\
 \item $\{\nu^{-1}_n\}_{n \in M\text{\textbackslash}\{l_n\}_{n \in \mathbb{N}}} \in l^1$ and\\
 \item $\displaystyle 0~ \textless ~\prod_{n \in \mathbb{N}} \frac{\nu_{l_n}}{\lambda_{k_n}}~ \textless ~\infty$.
\end{itemize} 
Then $\{\lambda_n\}_{n \in M}$, $\{\nu_n\}_{n \in M\text{\textbackslash}\{l_n\}_{n\in \mathbb{N}}}$, 
$\{\gamma_{k_n}(h_1)\}_{n \in \mathbb{N}}$, $h_1$(respectively $h_2$) and $\nu_m$ for some 
$m \in \{l_n\}_{n \in \mathbb{N}}$ uniquely determine the operator $J(g)$, $h_2$(respectively $h_1$) and if 
$J_{min} \neq J_{min}^*$, the boundary condition $g$ at infinity, where $\{\gamma_n(h_1)\}_{n \in M}$ are norming constants corresponding to $J_{h_1}(g)$. 
\end{theorem}

\begin{proof}
Following the proof of Theorem \ref{nIP1} we get representation
\begin{equation}\label{RepG5}
 \left(\prod_{n \in M}\frac{\nu_n}{\lambda_n}\right)\mathcal{G}(z,g) = C + \sum_{n\in \mathbb{N}}\frac{A_{k_n}}{z-\lambda_{k_n}},
\end{equation}
for the infinite product
\begin{equation*}
 \mathcal{G}(z,g) := \frac{1}{\Delta}\prod_{n \in \mathbb{N}}\frac{\nu_{l_n}-z}{\nu_{l_n}}\frac{\lambda_{k_n}}{\lambda_{k_n}-z},
\end{equation*}
where $A_j$ is the residue of $\displaystyle\left(\prod_{n \in M}\frac{\nu_n}{\lambda_n}\right)\mathcal{G}(z,g)$ at $\lambda_j$ and $C \in \mathbb{R}$.

Now let us prove uniqueness of $\displaystyle\left(\prod_{n \in M}\frac{\nu_n}{\lambda_n}\right)\mathcal{G}(z,g)$. Note that we know 
$\{\lambda_{k_n}\}_{n \in \mathbb{N}}$, $\{-A_{k_n}\}_{n \in \mathbb{N}}$ and 
$\nu_m$. Let the infinite product
\begin{equation*}
 \left(\prod_{n \in M}\frac{\widetilde{\nu}_n}{\lambda_n}\right)\widetilde{\mathcal{G}}(z,g) := \left(\prod_{n \in M}\frac{\widetilde{\nu}_n}{\lambda_n}\right)
 \frac{1}{\Delta}\prod_{n \in \mathbb{N}}\frac{\widetilde{\nu}_{l_n}-z}{\widetilde{\nu}_{l_n}}\frac{\lambda_{k_n}}{\lambda_{k_n}-z}
\end{equation*}
share the same set of poles $\{\lambda_{k_n}\}_{n \in \mathbb{N}}$ and the same residues $\{-A_{k_n}\}_{n \in \mathbb{N}}$ at the corresponding poles 
with $\displaystyle\left(\prod_{n \in M}\frac{\nu_n}{\lambda_n}\right)\mathcal{G}(z,g)$. 
In addition assume $\widetilde{\nu}_j = \nu_j$ for all $j \in M\text{\textbackslash}\{l_n\}_{n\in \mathbb{N}}$ and $\mathcal{G}(z,g)$ and 
$\widetilde{\mathcal{G}}(z,g)$ have the common zero $\nu_m$, i.e. $\widetilde{\nu}_m = \nu_m$.
Let us also assume zeros and poles 
of $\widetilde{\mathcal{G}}(z,g)$ satisfy asymptotic properties of Theorem \ref{C2S}. 
Then we know that $\widetilde{\mathcal{G}}(z,g)$ has the infinite sum representation
\begin{equation}\label{RepG_5}
 \left(\prod_{n \in M}\frac{\widetilde{\nu}_n}{\lambda_n}\right)\widetilde{\mathcal{G}}(z,g) = 
 \widetilde{C} + \sum_{n\in \mathbb{N}}\frac{A_{k_n}}{z-\lambda_{k_n}}
\end{equation}
From (\ref{RepG5}) and (\ref{RepG_5}), the difference of 
$\displaystyle\left(\prod_{n \in M}\frac{\nu_n}{\lambda_n}\right)\mathcal{G}(z,g)$ and 
$\displaystyle\left(\prod_{n \in M}\frac{\widetilde{\nu}_n}{\lambda_n}\right)\widetilde{\mathcal{G}}(z,g)$ is a real constant, 
which is zero since 
$\mathcal{G}(\nu_m,g)=\widetilde{\mathcal{G}}(\nu_m,g)=0$. This implies uniqueness of $\displaystyle\left(\prod_{n \in M}\frac{\nu_n}{\lambda_n}\right)\mathcal{G}(z,g)$ 
and hence uniqueness of $\{\nu_{l_n}\}_{n\in \mathbb{N}}$. After unique recovery of the two spectra $\sigma(J_{h_1}(g))$ and 
$\sigma(J_{h_2}(g))$, the operator $J$ is uniquely determined by Theorem \ref{2ST}.
\end{proof}

\begin{theorem}\label{nIP3} \emph{\textbf{(Inverse Problem VI)}} Let $J(g)$ be the Jacobi operator with discrete spectrum, 
$\sigma(J_{h_1}(g)) = \{\lambda_n\}_{n \in M}$, $\sigma(J_{h_2}(g)) = \{\nu_n\}_{n \in M}$ and 
$\displaystyle\{\lambda_{k_n}\}_{n\in \mathbb{N}} \subset \sigma(J_{h_1}(g))$, 
$\displaystyle \{\nu_{l_n}\}_{n\in \mathbb{N}} \subset \sigma(J_{h_2}(g))$ such that
\begin{itemize}
 \item $\displaystyle \inf_{n \in \mathbb{N}} |\nu_{l_n} - \lambda_{k_n}| ~\textgreater~ 0$,\\
 \item $\displaystyle \frac{A_{k_n,m} - A_{k_n}}{\lambda_{k_n}} \in l^1$ for all $n \in \{1,\dots,m\}$, for all $m \in \mathbb{N}$,\\
 \item $\{\lambda^{-1}_n\}_{n \in M} \in l^1$,\\
 \item $\{\nu^{-1}_n\}_{n \in M\text{\textbackslash}\{l_n\}_{n \in \mathbb{N}}} \in l^1$ and\\
 \item $\displaystyle 0~ \textless ~\prod_{n \in \mathbb{N}} \frac{\nu_{l_n}}{\lambda_{k_n}}~ \textless ~\infty$.
\end{itemize} 
Then $\{\lambda_n\}_{n \in M}$, $\{\nu_n\}_{n \in M\text{\textbackslash}\{l_n\}_{n\in \mathbb{N}}}$, 
$\{\gamma_{k_n}(h_1)\}_{n \in \mathbb{N}}$, $h_1$(respectively $h_2$) and $\gamma_m(h_1)$ for some 
$m \in M\text{\textbackslash}\{k_n\}_{n \in \mathbb{N}}$ uniquely determine the operator $J(g)$, $h_2$(respectively $h_1$) and 
if $J_{min} \neq J_{min}^*$, the boundary condition $g$ at infinity, where $\{\gamma_n(h_1)\}_{n \in M}$ are norming constants corresponding 
to $J_{h_1}(g)$. 
\end{theorem}

\begin{proof}
 Following the proof of Theorem \ref{nIP1} and redefining $\mathcal{G}$ and $\mathcal{H}$ as 
\begin{equation*}
 \mathcal{G}(z,g) := \frac{1}{\Delta}\left(\frac{\nu_m -z}{\nu_m}\frac{\lambda_m}{\lambda_m -z}\right)\prod_{n \in \mathbb{N}}\frac{\nu_{l_n}-z}{\nu_{l_n}}\frac{\lambda_{k_n}}{\lambda_{k_n}-z}
\end{equation*}
and
\begin{equation*}
 \mathcal{H}(z,g) := \left(\frac{\nu_m -z}{\nu_m}\frac{\lambda_m}{\lambda_m -z}\right)^{-1}
 \prod_{n \in M\text{\textbackslash}\{l_n\}}\frac{\nu_n-z}{\nu_n}\prod_{n \in M\text{\textbackslash}\{k_n\}}\frac{\lambda_n}{\lambda_n-z}
\end{equation*}
we get 
\begin{equation}\label{RepG6}
 \left(\prod_{n \in M}\frac{\nu_n}{\lambda_n}\right)\mathcal{G}(z,g) = C + \frac{A_m}{z-\lambda_m} + \sum_{n\in \mathbb{N}}\frac{A_{k_n}}{z-\lambda_{k_n}},
\end{equation}
where $A_j$ is the residue of $\displaystyle\left(\prod_{n \in M}\frac{\nu_n}{\lambda_n}\right)\mathcal{G}(z,g)$ at $\lambda_j$ and $C \in \mathbb{R}$.

Now let us prove uniqueness of $\displaystyle\left(\prod_{n \in M}\frac{\nu_n}{\lambda_n}\right)\mathcal{G}(z,g)$. Note that we know 
$\{\lambda_{k_n}\}_{n \in \mathbb{N}}\cup\{\lambda_m\}$, $\{A_{k_n}\}_{n \in \mathbb{N}}\cup\{A_m\}$ and $\nu_m$. 
Let the infinite product
\begin{equation*}
 \left(\prod_{n \in M}\frac{\widetilde{\nu}_n}{\lambda_n}\right)\widetilde{\mathcal{G}}(z,g) := \left(\prod_{n \in M}\frac{\widetilde{\nu}_n}{\lambda_n}\right)
 \frac{1}{\Delta}\left(\frac{\nu_m -z}{\nu_m}\frac{\lambda_m}{\lambda_m -z}\right)\prod_{n \in \mathbb{N}}\frac{\nu_{l_n}-z}{\nu_{l_n}}\frac{\lambda_{k_n}}{\lambda_{k_n}-z}
\end{equation*}
share the same set of poles $\{\lambda_{k_n}\}_{n \in \mathbb{N}}\cup\{\lambda_m\}$ and the same residues $\{A_{k_n}\}_{n \in \mathbb{N}}\cup\{A_m\}$ at 
the corresponding poles with $\mathcal{G}(z,g)$. In addition assume $\widetilde{\nu}_j = \nu_j$ for all $j \in M\text{\textbackslash}\{l_n\}_{n\in \mathbb{N}}$. 
Let us also assume zeros and poles of $\widetilde{\mathcal{G}}(z,g)$ satisfy asymptotic properties of Theorem \ref{C2S}. Then we know that 
$\widetilde{\mathcal{G}}(z,g)$ has the infinite sum representation
\begin{equation}\label{RepG_6}
 \left(\prod_{n \in M}\frac{\widetilde{\nu}_n}{\lambda_n}\right)\widetilde{\mathcal{G}}(z,g) = \widetilde{C} + \frac{A_m}{z-\lambda_m} + \sum_{n\in \mathbb{N}}\frac{A_{k_n}}{z-\lambda_{k_n}}
\end{equation}
From (\ref{RepG6}) and (\ref{RepG_6}), the difference of 
$\displaystyle\left(\prod_{n \in M}\frac{\nu_n}{\lambda_n}\right)\mathcal{G}(z,g)$ and 
$\displaystyle\left(\prod_{n \in M}\frac{\widetilde{\nu}_n}{\lambda_n}\right)\widetilde{\mathcal{G}}(z,g)$ is a real constant, 
which is zero since 
$\mathcal{G}(\nu_m,g)=\widetilde{\mathcal{G}}(\nu_m,g)=0$. This implies uniqueness of $\displaystyle\left(\prod_{n \in M}\frac{\nu_n}{\lambda_n}\right)\mathcal{G}(z,g)$ 
and hence uniqueness of $\{\nu_{l_n}\}_{n\in \mathbb{N}}$. After unique recovery of the two spectra $\sigma(J_{h_1}(g))$ and 
$\sigma(J_{h_2}(g))$, the operator $J$ is uniquely determined by Theorem \ref{2ST}.

\end{proof}

\end{document}